\documentclass{article}%
\usepackage{amssymb}
\usepackage{graphicx}
\usepackage{amsmath}
\usepackage{amsfonts}%
\newtheorem{theorem}{Theorem}

\newtheorem{corollary}[theorem]{Corollary}

\newtheorem{definition}[theorem]{Definition}

\newtheorem{lemma}[theorem]{Lemma}

\newtheorem{proposition}[theorem]{Proposition}

\newenvironment{proof}[1][Proof]{\textbf{#1.} }{\ \rule{0.5em}{0.5em}}
\begin{document}

\title{Weighted interlace polynomials}
\author{Lorenzo Traldi\\Lafayette College\\Easton, Pennsylvania 18042}
\date{}
\maketitle

\begin{abstract}
The interlace polynomials introduced by Arratia, Bollob\'{a}s and Sorkin
extend to invariants of graphs with vertex weights, and these weighted
interlace polynomials have several novel properties. One novel property is a
version of the fundamental three-term formula
\[
q(G)=q(G-a)+q(G^{ab}-b)+((x-1)^{2}-1)q(G^{ab}-a-b)
\]
that lacks the last term. It follows that interlace polynomial computations
can be represented by binary trees rather than mixed binary-ternary trees.
Binary computation trees provide a description of $q(G)$ that is analogous to
the activities description of the Tutte polynomial. If $G$ is a tree or forest
then these ``algorithmic activities''\ are associated with a certain kind of
independent set in $G$. Three other novel properties are weighted pendant-twin
reductions, which involve removing certain kinds of vertices from a graph and
adjusting the weights of the remaining vertices in such a way that the
interlace polynomials are unchanged. These reductions allow for smaller
computation trees as they eliminate some branches. If a graph can be
completely analyzed using pendant-twin reductions then its interlace
polynomial can be calculated in polynomial time. An intuitively pleasing
property is that graphs which can be constructed through graph substitutions
have vertex-weighted interlace polynomials which can be obtained through
algebraic substitutions.

\bigskip

Keywords. interlace polynomial, vertex weight, pendant vertex, twin vertex,
series, parallel, graph composition, graph substitution, join, computational
complexity, tree, Tutte polynomial, Jones polynomial

\bigskip

Mathematics Subject\ Classification. 05C50

\end{abstract}

\section{Introduction}

Motivated by problems arising from DNA sequencing and the properties of circle
graphs of 2-in, 2-out digraphs, Arratia, Bollob\'{a}s and Sorkin introduced a
new family of graph invariants, the interlace polynomials, in \cite{A1, A2,
A}. These invariants may be defined either through recursive elimination of
vertices or as sums indexed by subsets of vertices \cite{AH, A}, much as the
Tutte polynomial may be defined either through recursive elimination of edges
or as a sum indexed by subsets of edges. Ellis-Monaghan and Sarmiento
\cite{EMS} have shown (among other results) that one of the one-variable
interlace polynomials, $q_{N}$, can be computed in polynomial time for
bipartite distance hereditary graphs. Their proof depends on the corresponding
result for the Tutte polynomials of series-parallel graphs, first proved in
\cite{OW}, and the fact that bipartite distance hereditary graphs are circle
graphs of Euler circuits in medial graphs of series-parallel graphs.

\bigskip

In this paper we discuss several useful features of interlace polynomials that
have been modified to incorporate vertex weights. After defining the weighted
interlace polynomials in Section 2, we observe that a simple adjustment of
weights makes it unnecessary to have the third term in the fundamental
recursion $q(G)=q(G-a)+q(G^{ab}-b)+((x-1)^{2}-1)q(G^{ab}-a-b)$ of \cite{A}. In
Section 3 we present reduction formulas that can be used to eliminate a vertex
that is a twin of another, or pendant on another. These pendant-twin
reductions are analogous to the series-parallel reductions of electrical
circuit theory, in which two resistors wired in series (resp. parallel) are
replaced by one resistor with $R=R_{1}+R_{2}$ (resp. $R^{-1}=R_{1}^{-1}%
+R_{2}^{-1}$). The pendant-twin reductions are used to extend the result of
Ellis-Monaghan and Sarmiento mentioned above to the two-variable interlace
polynomials of looped, non-bipartite distance hereditary graphs. In Section 4
we show that if $G=H\ast K$ is a looped graph obtained using the composition
construction of Cunningham \cite{Cu} then $q(G)$ is equal to the interlace
polynomial of a suitably\ re-weighted version of $K$; this generalizes results
of \cite{A2} that describe the $q_{N}$ polynomials of simple graphs
constructed through substitution. Composition has proven useful in the study
of circle graphs (see for instance \cite{Bd, Be, Cj}), so it is not surprising
to see it appear in the theory of the interlace polynomials. In\ Section 5 we
discuss some elementary properties of the unweighted $q_{N}$ polynomial,
focusing on simple (unlooped) graphs. In Section 6 we sketch a combinatorial
description of the interlace polynomials of trees and forests introduced by
Anderson, Cutler, Radcliffe and the present author in \cite{ACRT}. This
combinatorial description bears a striking resemblance to the activities
description of the Tutte polynomial, and in\ Section 7 we extend it to
arbitrary graphs using \textquotedblleft activities\textquotedblright\ defined
with respect to recursive interlace polynomial calculations. We do not know
whether or not these activities have a convenient combinatorial description in general.

\bigskip

We should observe that the idea of using vertex weights for interlace
polynomials has appeared before, though our implementation of the idea is
different from those we have seen elsewhere. In \cite{Ci}, Courcelle
introduced a multivariate interlace polynomial that is more complicated than
the polynomials we consider here, and involves assigning indeterminates to the
vertices of a graph. He used monadic second-order logic to show that it is
possible to compute bounded portions of this polynomial (and the entire
unweighted interlace polynomial $q$) in polynomial time for graphs of bounded
clique-width. This technique is quite general but involves large built-in
constants, so for pendant-twin reductions and compositions the formulas
presented here are considerably more practical. Also, Bl\"{a}ser and Hoffmann
\cite{BH} use the idea of assigning indeterminates to vertices, along with the
adjunction of two types of pendant-twin vertices, to show that evaluating
interlace polynomials is generally $\#P$-hard for almost all values of the variables.

\bigskip

In \cite{A1} Arratia, Bollob\'{a}s and Sorkin observe that there is a natural
(so natural it is ``practically a tautology'') correspondence between the
Kauffman bracket of an alternating link diagram and an interlace polynomial of
an associated 2-in, 2-out digraph. The situation is clear enough that we do
not discuss it in detail, but it is worth mentioning that this correspondence
may be extended to arbitrary link diagrams using vertex weights. The
well-known relationship between the Jones and Tutte polynomials is similar, in
that an edge-weighted or -signed version of the Tutte polynomial conveniently
incorporates crossing information when dealing with non-alternating links
\cite{K1, M, T1}.

\section{Expansions and recursions}

We recall terminology and notation of \cite{A}. \textit{Graphs} may have loops
but not multiple edges or multiple loops. The \textit{rank} $r(G)$ and
\textit{nullity} $n(G)$ of a graph $G$ are those of its adjacency matrix
considered over $GF(2)$. If $S\subseteq V(G)$ then $G[S]$ is the subgraph of
$G$ induced by $S$. In addition, we say a graph is (\textit{vertex-})
\textit{weighted} by functions $\alpha$ and $\beta$ mapping $V(G)$ into some
commutative ring with unity $R$. For ease of notation we prefer to denote a
weighted graph $G$ rather than using the triple $(G,\alpha,\beta)$; even when
two graphs differ only in their weights we will denote them $G$ and
$G^{\prime}$ rather than $(G,\alpha,\beta)$ and $(G,\alpha^{\prime}%
,\beta^{\prime})$. Also, if a graph is modified then unless otherwise stated,
we presume that the weight functions $\alpha$ and $\beta$ are modified in the
most natural way. For instance, if $a\in V(G)$ then the other vertices of $G$
have the same vertex weights in $G-a$ as they have in $G$. If $G$ is a
weighted graph then the \textit{unweighted version} of $G$ is denoted $G^{u}$;
it has the same underlying graph and the trivial weights $\alpha\equiv
\beta\equiv1\in\mathbb{Z}$.

\begin{definition}
\label{2varq}If $G$ is a vertex-weighted graph then the \emph{weighted
interlace polynomial} of $G$ is
\[
q(G)=\sum_{S\subseteq V(G)}(%
%TCIMACRO{\dprod \limits_{s\in S}}%
%BeginExpansion
{\displaystyle\prod\limits_{s\in S}}
%EndExpansion
\alpha(s))(%
%TCIMACRO{\dprod \limits_{v\not \in S}}%
%BeginExpansion
{\displaystyle\prod\limits_{v\not \in S}}
%EndExpansion
\beta(v))(x-1)^{r(G[S])}(y-1)^{n(G[S])}.
\]

\end{definition}

\begin{definition}
If $G$ is a vertex-weighted graph then the \emph{weighted vertex-nullity
interlace polynomial} of $G$ is
\[
q_{N}(G)=\sum_{S\subseteq V(G)}(%
%TCIMACRO{\dprod \limits_{s\in S}}%
%BeginExpansion
{\displaystyle\prod\limits_{s\in S}}
%EndExpansion
\alpha(s))(%
%TCIMACRO{\dprod \limits_{v\not \in S}}%
%BeginExpansion
{\displaystyle\prod\limits_{v\not \in S}}
%EndExpansion
\beta(v))(y-1)^{n(G[S])}.
\]

\end{definition}

\begin{definition}
If $G$ is a vertex-weighted graph then the \emph{weighted vertex-rank
interlace polynomial} of $G$ is
\[
q_{R}(G)=\sum_{S\subseteq V(G)}(%
%TCIMACRO{\dprod \limits_{s\in S}}%
%BeginExpansion
{\displaystyle\prod\limits_{s\in S}}
%EndExpansion
\alpha(s))(%
%TCIMACRO{\dprod \limits_{v\not \in S}}%
%BeginExpansion
{\displaystyle\prod\limits_{v\not \in S}}
%EndExpansion
\beta(v))(x-1)^{r(G[S])}.
\]

\end{definition}

The original, unweighted interlace polynomials of $G$ are recovered by using
$G^{u}.$

\bigskip

The three weighted interlace polynomials of a particular weighted graph $G$
may be substantially different from each other. Considered as functions
defined on the class of all weighted graphs, however, the three polynomials
are essentially equivalent: if $\tilde{G}$ is obtained from $G$ by
re-weighting $V(G)=\{v_{1},...,v_{n}\}$ using the indeterminates in the
polynomial ring $\mathbb{Z}[\alpha_{1},..,\alpha_{n},\beta_{1},...,\beta_{n}%
]$, then the value of any one weighted interlace polynomial on $\tilde{G} $
determines the values of all three polynomials on all weighted versions of
$G$. For instance, $q(G)$ can be obtained from $q_{N}(\tilde{G})$ by
substituting $(x-1)\cdot\alpha(v_{i})$ for $\alpha_{i}$, $\beta(v_{i})$ for
$\beta_{i}$ and $1+\frac{y-1}{x-1}$ for $y$. (In contrast, the functions
defined by the three unweighted interlace polynomials are rather different
from one another, as there are many pairs of graphs distinguished by $q_{R}$
but not by $q_{N}$ \cite{A}.) We mention all three weighted polynomials simply
because one of them may be more convenient for some purposes than the others.
Most of our results are stated for $q$ because it specializes to the others
most readily.

\bigskip

In fact any one of $q(\tilde{G})$, $q_{N}(\tilde{G})$, $q_{R}(\tilde{G})$
determines $G$ up to isomorphism, because the vertex weights identify the
contribution of each $S\subseteq V(G)$. The looped vertices of $G$ appear in
the 1-element subsets $S$ of rank 1 (nullity 0), some pairs of adjacent
vertices of $G$ appear in the 2-element subsets $S$ that contain at least one
unlooped vertex and have rank 2 (nullity 0), and the other pairs of adjacent
vertices of $G$ appear in the 2-element subsets $S$ that contain two looped
vertices and have rank 1 (nullity 1). Moreover the same comment applies even
if only one of $\alpha,\beta$ is nontrivial, i.e., if $\beta\equiv1$ or
$\alpha\equiv1$. This might make it seem possible to simplify our discussion
by considering only $\alpha$ or only $\beta$, but we prefer to use both
weights because they produce especially easy-to-read formulas, as noted in the
discussions of Theorems \ref{theorem1} and \ref{theorem3} below. Setting one
weight identically equal to 1 would not simply leave us with the other weight;
in essence, the remaining one would replace the ratio of the original two.
Consequently, using only $\alpha$ or only $\beta$ would entail unnecessary
algebraic complications and losses of generality, because many formulas would
require division. For instance, with $\beta\equiv1$ Proposition \ref{prop1}
would state that replacing $\alpha(a)$ with $\alpha^{\prime}(a)=r_{1}%
\alpha(a)/(r_{1}+r_{2})$ results in a graph $G^{\prime}$ with $(r_{1}%
+r_{2})q(G^{\prime})=r_{1}q(G)+r_{2}q(G-a)$. The special case $r_{1}=-r_{2}$
would require a separate statement, and of course $\beta(a)=0$ would be ruled out.

\bigskip

Our first proposition follows immediately from Definition \ref{2varq}.

\begin{proposition}
\label{prop1}Suppose $a\in V(G)$ and $r_{1},r_{2}\in R$. Let $G^{\prime}$ be
obtained from $G$ by changing the weights of $a$ to $\alpha^{\prime}%
(a)=r_{1}\alpha(a)$ and $\beta^{\prime}(a)=r_{1}\beta(a)+r_{2}$. Then
$q(G^{\prime})=r_{1}q(G)+r_{2}q(G-a)$.
\end{proposition}

A fundamental property of the interlace polynomials is that they can be
calculated recursively with the local complementation and pivot operations
used by Kotzig \cite{Ko}, Bouchet \cite{Bc, B} and Arratia, Bollob\'{a}s and
Sorkin \cite{A1, A2, A}.

\begin{definition}
(Local Complementation) If $a$ is a vertex of $G$ then $G^{a}$ is obtained
from $G$ by toggling\ adjacencies $\{x,y\}$ involving neighbors of $a$ that
are distinct from $a$.
\end{definition}

\begin{definition}
(Pivot) If $a$ and $b$ are distinct vertices of $G$ then the graph $G^{ab}$ is
obtained from $G$ by toggling\ adjacencies $\{x,y\}$ such that $x,y\notin
\{a,b\}$, $x$ is adjacent to $a$ in $G$, $y$ is adjacent to $b$ in $G$, and
either $x$ is not adjacent to $b$ or $y$ is not adjacent to $a$.
\end{definition}

Note that local complementation includes loop-toggling (when $x=y$ is a
neighbor of $a$ distinct from $a$) but pivoting does not.

\begin{theorem}
\label{theorem1}If $G$ is a weighted graph then $q(G)$ can be calculated
recursively using the following properties.

(a) If $a$ is a looped vertex then
\[
q(G)=\beta(a)q(G-a)+\alpha(a)(x-1)q(G^{a}-a).
\]

(b) If $a$ and $b$ are loopless neighbors then
\[
q(G)=\beta(a)q(G-a)+\beta(b)q(G^{ab}-b)+(\alpha(a)\alpha(b)(x-1)^{2}%
-\beta(a)\beta(b))q(G^{ab}-a-b).
\]

(c) If $G$ has no non-loop edges then $q(G)$ is
\[
\left(
%TCIMACRO{\dprod \limits_{\mathrm{unlooped}~v\in V(G)}}%
%BeginExpansion
{\displaystyle\prod\limits_{\mathrm{unlooped}~v\in V(G)}}
%EndExpansion
\left(  \alpha(v)(y-1)+\beta(v)\right)  \right)  \cdot\left(
%TCIMACRO{\dprod \limits_{\mathrm{looped}~v\in V(G)}}%
%BeginExpansion
{\displaystyle\prod\limits_{\mathrm{looped}~v\in V(G)}}
%EndExpansion
\left(  \alpha(v)(x-1)+\beta(v)\right)  \right)  .
\]

\end{theorem}

\begin{proof}
The proofs of parts (a) and (b) of Theorem \ref{theorem1} are essentially the
same as the proofs of the corresponding formulas for the unweighted
two-variable interlace polynomial \cite{A}. Indeed, if we read $\alpha$ as
\textquotedblleft includes\textquotedblright\ and $\beta$ as \textquotedblleft
excludes\textquotedblright\ then\ the weighted formulas serve as mnemonic
devices to recall the proofs.

For instance part (a) is proven as follows. The term $\beta(a)q(G-a)$ reflects
the fact that if $a\not \in S\subseteq V(G)$ then $G[S]=(G-a)[S]$, so the
contributions of $S$ to $q(G)$ and $q(G-a)$ differ only by a factor $\beta
(a)$. The term $\alpha(a)(x-1)q(G^{a}-a)$ reflects the fact that if $a\in
S\subseteq V(G)$ then
\[
r(G[S])=r
\begin{pmatrix}
1 & \mathbf{1} & \mathbf{0}\\
\mathbf{1} & M_{11} & M_{12}\\
\mathbf{0} & M_{21} & M_{22}%
\end{pmatrix}
=1+r
\begin{pmatrix}
M_{11}^{c} & M_{12}\\
M_{21} & M_{22}%
\end{pmatrix}
=1+r((G^{a}-a)[S-a]),
\]
where bold numerals represent rows and columns and $M_{11}^{c}$ differs from
$M_{11}$ in every entry. Consequently the contributions of $S$ to $q(G)$ and
$S-\{a\}$ to $q(G^{a}-a)$ differ by a factor $\alpha(a)(x-1)$.

The more complicated formula of part (b) reflects, among other things, the
fact that each subset $S\subseteq V(G)$ with $a,b\not \in S$ contributes to
all three terms; the last two contributions cancel each other.

Part (c) follows directly from\ Definition \ref{2varq}. It is not technically
necessary to mention graphs with loops in (c), as loops can always be removed
using part (a). However applying part (a) to completely disconnected graphs is
obviously inefficient.
\end{proof}

\bigskip

Other properties of the unweighted interlace polynomials also extend naturally
to the weighted polynomials. For instance, the second and third parts of
Theorem \ref{theorem2} below extend two properties discussed in Section 3 of
\cite{A}. The first part will be useful in Section 4, where we discuss
substituted graphs.

\bigskip

\begin{theorem}
\label{theorem2}(a) If $a\in V(G)$ is loopless then
\[
q(G)-\beta(a)q(G-a)=q(G^{a})-\beta(a)q(G^{a}-a).
\]

(b) If $a,b\in V(G)$ are loopless neighbors then
\[
q(G-a)-\beta(a)q(G-a-b)=q(G^{ab}-a)-\beta(a)q(G^{ab}-a-b).
\]

(c) If $G$ is the union of disjoint subgraphs $G_{1}$ and $G_{2}$ then
$q(G)=q(G_{1})q(G_{2})$.
\end{theorem}

\begin{proof}
To prove (a), note that if $a\in S\subseteq V(G)$ then the adjacency matrix of
$G[S]$ may be represented by
\[%
\begin{pmatrix}
0 & \mathbf{1} & \mathbf{0}\\
\mathbf{1} & M_{11} & M_{12}\\
\mathbf{0} & M_{21} & M_{22}%
\end{pmatrix}
.
\]
Adding the first row to each of those in the second group results in
\[%
\begin{pmatrix}
0 & \mathbf{1} & \mathbf{0}\\
\mathbf{1} & M_{11}^{c} & M_{12}\\
\mathbf{0} & M_{21} & M_{22}%
\end{pmatrix}
,
\]
\noindent the adjacency matrix of $G^{a}[S]$. Definition \ref{2varq} then
tells us that

\smallskip%
\begin{align*}
&  q(G)-\beta(a)q(G-a)\\
&  ~\\
&  =\sum_{a\in S\subseteq V(G)}(%
%TCIMACRO{\dprod \limits_{s\in S}}%
%BeginExpansion
{\displaystyle\prod\limits_{s\in S}}
%EndExpansion
\alpha(s))(%
%TCIMACRO{\dprod \limits_{v\not \in S}}%
%BeginExpansion
{\displaystyle\prod\limits_{v\not \in S}}
%EndExpansion
\beta(v))(x-1)^{r(G[S])}(y-1)^{n(G[S])}\\
&  ~\\
&  =\sum_{a\in S\subseteq V(G)}(%
%TCIMACRO{\dprod \limits_{s\in S}}%
%BeginExpansion
{\displaystyle\prod\limits_{s\in S}}
%EndExpansion
\alpha(s))(%
%TCIMACRO{\dprod \limits_{v\not \in S}}%
%BeginExpansion
{\displaystyle\prod\limits_{v\not \in S}}
%EndExpansion
\beta(v))(x-1)^{r(G^{a}[S])}(y-1)^{n(G^{a}[S])}\\
&  ~\\
&  =q(G^{a})-\beta(a)q(G^{a}-a).
\end{align*}

\smallskip

The proofs of (b) and (c) are essentially the same as the proofs of the
corresponding results in \cite{A}.
\end{proof}

\bigskip

Many properties of the unweighted interlace polynomials extend directly to the
weighted polynomials, as we see in Theorems \ref{theorem1} and \ref{theorem2}.
It may be a surprise that some properties of the unweighted interlace
polynomials can be significantly simplified using vertex weights. For
instance, consider a computation tree representing a recursive implementation
of Theorem \ref{theorem1}. The tree has two branches for each application of
part (a), three branches for each application of part (b), and a leaf for each
application of part (c). As noted in the proof of Theorem \ref{theorem1}, the
three terms of part (b) all incorporate contributions from the same subsets
$S\subseteq V(G)$. The recursive computation calculates these same
contributions three times, on separate branches. This inefficiency can be
eliminated by rephrasing part (b) of Theorem \ref{theorem1} so that no
three-fold branches are necessary.

\begin{corollary}
\label{cor1}If $a$ and $b$ are loopless neighbors in~$G$ then $q(G)=\beta
(a)q(G-a)+\alpha(a)q((G^{ab}-b)^{\prime})$, where $(G^{ab}-b)^{\prime}$ is
obtained from $G^{ab}-b$ by changing the weights of $a$ to $\alpha^{\prime
}(a)=\beta(b)$ and $\beta^{\prime}(a)=\alpha(b)(x-1)^{2}$.
\end{corollary}

\begin{proof}
Let $(G^{ab}-b)^{\prime\prime}$ be the graph obtained from $G^{ab}-b$ by
changing the weights of $a$ as in Proposition \ref{prop1}, with $r_{1}%
=\beta(b)$ and $r_{2}=\alpha(a)\alpha(b)(x-1)^{2}$ -- $\beta(a)\beta(b)$. Then
$\alpha^{\prime\prime}(a)=a(a)\beta(b)$ and $\beta^{\prime\prime}%
(a)=\alpha(a)\alpha(b)(x-1)^{2}$, so $q((G^{ab}-b)^{\prime\prime})$ =
$\alpha(a)q((G^{ab}-b)^{\prime})$. On the other hand, Proposition \ref{prop1}
tells us that
\[
q((G^{ab}-b)^{\prime\prime})=\beta(b)q(G^{ab}-b)+(\alpha(a)\alpha
(b)(x-1)^{2}-\beta(a)\beta(b))q(G^{ab}-b-a).
\]

\end{proof}

\bigskip

If $\alpha^{\prime}(a)=\beta(b)/(x-1)$ and $\beta^{\prime}(a)=\alpha(b)(x-1)$
are used instead of the weights given in Corollary \ref{cor1}, then the
resulting formula $q(G)=\beta(a)q(G-a)+\alpha(a)(x-1)q((G^{ab}-b)^{\prime})$
still has only two branches, and bears an interesting resemblance to part (a)
of Theorem \ref{theorem1}. However using division to define $\alpha^{\prime
}(a)$ may occasionally cause some algebraic difficulties, e.g., it complicates
the evaluation at $x=1$, and it prohibits the use of rings in which $x-1$ is a
divisor of zero.

Corollary \ref{cor1} is one of several results in which vertex weights allow
us to extend properties of the unweighted version of $q_{N}$, seemingly the
simplest kind of interlace polynomial, to the other interlace polynomials. In
this instance the extended property is that of possessing a recursive
description represented by a binary computation tree. Another result of this
type is Corollary \ref{cor2}, which extends Remark 18 of \cite{A2}: if $a$ and
$b$ are loopless neighbors in $G$ then the unweighted vertex-nullity interlace
polynomials of $G$ and $G^{ab}$ are the same.

\begin{corollary}
\label{cor2}Let $a$ and $b$ be unlooped neighbors in $G$, and let
$(G^{ab})^{\prime}$ be the weighted graph obtained from $G^{ab}$ by changing
the weights of $a$ and $b$ to $\alpha^{\prime}(a)=\beta(b)$, $\beta^{\prime
}(a)=\alpha(b)(x-1)^{2}$, $\alpha^{\prime}(b)=\beta(a)$ and $\beta^{\prime
}(b)=\alpha(a)(x-1)^{2}$. Then $(x-1)^{2}q(G)=q((G^{ab})^{\prime})$.
\end{corollary}

\begin{proof}
Applying Corollary \ref{cor1} to $(G^{ab})^{\prime}$, with the roles of $a$
and $b$ reversed, tells us that $q((G^{ab})^{\prime})$ = $\beta^{\prime
}(b)q((G^{ab})^{\prime}-b)+\alpha^{\prime}(b)q((((G^{ab})^{\prime}%
)^{ab}-a)^{\prime})$. Observe that $(((G^{ab})^{\prime})^{ab}-a)^{\prime} $
has the underlying graph $(G^{ab})^{ab}-a=G-a$, and differs from $G-a$ only in
the weights of $b$, which are given by $\alpha^{\prime\prime}(b)=\beta
^{\prime}(a)=\alpha(b)(x-1)^{2}$ and $\beta^{\prime\prime}(b)=\alpha^{\prime
}(a)(x-1)^{2}=\beta(b)(x-1)^{2}$. Consequently $\alpha^{\prime}(b)q((((G^{ab}%
)^{\prime})^{ab}-a)^{\prime})$ = $\alpha^{\prime}(b)(x-1)^{2}q(G-a)$ =
$\beta(a)(x-1)^{2}q(G-a)$. As $\beta^{\prime}(b)q((G^{ab})^{\prime}-b)$ =
$\alpha(a)(x-1)^{2}q((G^{ab}-b)^{\prime})$, the result follows directly
from\ Corollary \ref{cor1}.
\end{proof}

If we are confident that division by $x-1$ will cause no trouble, Corollary
\ref{cor2} may be restated with a simpler conclusion: using $\alpha^{\prime
}(a)=\beta(b)/(x-1)$, $\beta^{\prime}(a)=\alpha(b)(x-1)$, $\alpha^{\prime
}(b)=\beta(a)/(x-1)$ and $\beta^{\prime}(b)=\alpha(a)(x-1)$ yields
$q(G)=q((G^{ab})^{\prime})$.

\section{Pendant-twin reductions}

Other novel properties of the weighted interlace polynomials arise from a
general observation. Suppose $a,b\in V(G)$ happen to have the property that
for $S\subseteq V(G)$, the rank and nullity of $S$ are determined by the rank
and nullity of $S-\{b\}$, perhaps in different ways according to which of
$a,b$ is contained in $S$. Then it may be possible to adjust $\alpha(a)$ and
$\beta(a)$ in such a way that the weighted interlace polynomial of the
weight-adjusted version of $G-b$ incorporates all the information in $q(G)$.
The simplest instance of this observation occurs when $a$ and $b$ give rise to
identical rows and columns in the adjacency matrix of $G$.

\begin{definition}
Two vertices $a,b$ of $G$ are\emph{\ identical twins }if (i) either they are
looped and adjacent or they are unlooped and not adjacent, and (ii) they have
the same neighbors outside $\{a,b\}$.
\end{definition}

Identical twins are called \textit{clones} in \cite{BH}, and unlooped
identical twins are called \textit{duplicates} in \cite{A2} and \textit{false
twins} in \cite{EMS}. We prefer the present terminology because the adjective
\textit{false} seems inappropriate, and because we do not know what we would
call non-identical duplicates or clones in Definition 7 below. The following
result extends Proposition 40 of \cite{A2}, Proposition 4.14 of \cite{EMS} and
Section 3.1 of \cite{BH} to vertex-weighted graphs.

\begin{theorem}
\label{theorem3}Suppose $a$ and $b$ are identical twins in $G$. Let
$G^{\prime}$ be the graph obtained from $G-b$ by changing the weights of $a$:
$\beta^{\prime}(a)=\beta(a)\beta(b)$ and $\alpha^{\prime}(a)=\alpha
(a)\beta(b)+\alpha(a)\alpha(b)(y-1)+\beta(a)\alpha(b)$. Then $q(G)=q(G^{\prime
})$.
\end{theorem}

\begin{proof}
If $a\not \in S\subseteq V(G^{\prime})$ then $r(G^{\prime}[S])$ = $r(G[S])$,
because the\ adjacency matrices are the same. If $a\in S\subseteq V(G^{\prime
})$ then $r(G^{\prime}[S])$ = $r(G[S])$ = $r(G[S\cup\{b\}])$ =
$r(G[(S-\{a\})\cup\{b\}])$, because the only difference among the adjacency
matrices is that the matrix of $G[S\cup\{b\}]$ has two identical rows and
columns, corresponding to $a$ and $b$. As mentioned above, it is helpful to
read $\alpha$ as \textquotedblleft includes\textquotedblright\ and $\beta$ as
\textquotedblleft excludes\textquotedblright\ so that (for instance) the
appearance of $\beta(a)\alpha(b)$ in $\alpha^{\prime}(a)$ indicates that if
$a\in S\subseteq V(G^{\prime})$\ then the summand of $q(G^{\prime})$
corresponding to $S$ includes the summand of $q(G)$ corresponding to
$(S-\{a\})\cup\{b\}$.
\end{proof}

Theorem \ref{theorem3} has the following inductive generalization. Suppose
$k\geq1$ and $a=b_{0},b_{1},...,b_{k}$ are identical twins in $G$. If
$G^{\prime}$ is the graph obtained from $G-b_{1}-...-b_{k}$ by changing the
weights of $a$ to
\begin{align*}
\alpha^{\prime}(a)  &  =\sum_{\emptyset\neq S\subseteq\{b_{0},...,b_{k}\}}(%
%TCIMACRO{\dprod \limits_{b_{i}\in S}}%
%BeginExpansion
{\displaystyle\prod\limits_{b_{i}\in S}}
%EndExpansion
\alpha(b_{i}))(%
%TCIMACRO{\dprod \limits_{b_{j}\not \in S}}%
%BeginExpansion
{\displaystyle\prod\limits_{b_{j}\not \in S}}
%EndExpansion
\beta(b_{j}))(y-1)^{\left|  S\right|  -1}\\
& \\
\mathrm{and}~\beta^{\prime}(a)  &  =%
%TCIMACRO{\dprod \limits_{i=0}^{k}}%
%BeginExpansion
{\displaystyle\prod\limits_{i=0}^{k}}
%EndExpansion
\beta(b_{i}),
\end{align*}
\noindent then $q(G)=q(G^{\prime})$. We refer to the process of
combining\ several identical twins into a single vertex as an
\textit{identical twin reduction} no matter how many\ identical twins are combined.

\begin{theorem}
\label{theorem4}Suppose $b$ is an unlooped degree-one vertex pendant on $a$.
Let $G^{\prime}$ be the graph obtained from $G-b$ by changing the weights of
$a$: $\alpha^{\prime}(a)=\alpha(a)\beta(b)$ and $\beta^{\prime}(a)=\alpha
(a)\alpha(b)(x-1)^{2}+\beta(a)\alpha(b)(y-1)+\beta(a)\beta(b)$. Then
$q(G)=q(G^{\prime})$.
\end{theorem}

\begin{proof}
If $a\in S\subseteq V(G^{\prime})$ then $r(G^{\prime}[S])$ = $r(G[S])$,\ and
if $a\not \in S\subseteq V(G^{\prime})$\ then $r(G^{\prime}[S])$ =
$r(G[S\cup\{a,b\}])-2$ = $r(G[S\cup\{b\}])$ = $r(G[S])$.
\end{proof}

If $b=b_{0},b_{1},...,b_{k}$ are unlooped and pendant on $a$ then $b_{0}%
,b_{1},...,b_{k}$ are identical twins, so we can combine them into a single
re-weighted vertex $b$ using Theorem \ref{theorem3} and then remove $b $ using
Theorem \ref{theorem4}. We refer to the removal of any number of unlooped
vertices pendant on the same vertex as an \textit{unlooped pendant vertex
reduction}. Theorem \ref{theorem4} extends Proposition 4.12 of \cite{EMS} and
Section 3.2 of \cite{BH}, where collections of pendant vertices are called
\textit{combs}.

\bigskip

A recursive calculation that depends solely on Theorem \ref{theorem1} and
Corollary \ref{cor1} is represented by a mixed binary-ternary computation
tree; if Corollary \ref{cor1} is always used in place of part (b) of Theorem
\ref{theorem1} then the tree will be binary. Using identical twin and unlooped
pendant vertex reductions makes the computation tree smaller, because each
time one of these reductions is used, we avoid splitting the resulting portion
of the calculation into branches. Similarly, a Tutte polynomial computation
that incorporates series-parallel reductions and deletion-contraction
operations will generally result in a smaller formula than a computation that
involves only deletion-contraction operations can provide. This latter
observation was made precise in \cite{T2}: computing the Tutte polynomial of a
matroid $M$ using series-parallel reductions and deletion-contraction
operations will result in an expression with at least $\beta(M)$ terms, where
$\beta(M)$ is Crapo's $\beta$ invariant \cite{Cr}; moreover if $\beta(M)>0$
the lower bound is attainable. (This result is merely the extension to the
Tutte polynomial of the important theory of reliability domination; see
\cite{BSS, C} for expositions.) For interlace polynomials, analogous lower
bounds are derived from the unweighted vertex-nullity polynomial $q_{N}%
(G^{u})$. The coefficient of $y$ in $q_{N}(G^{u})$ is denoted $\gamma(G)$ as
in \cite{EMS}, and the evaluation $q_{N}(G^{u})(0)$ is denoted $\varepsilon
(G)$.

\begin{theorem}
\label{theorem5}If part (c) of Theorem \ref{theorem1} is applied only to
loopless graphs then a computation of $q(G)$ using Corollary \ref{cor1} and
Theorems \ref{theorem1}, \ref{theorem3} and \ref{theorem4} is represented by a
computation tree with no fewer than $\frac{1}{2}\varepsilon(G)$ leaves. If $G$
is a simple graph then the computation tree has no fewer than $\frac{1}%
{2}\gamma(G)$ leaves.
\end{theorem}

\begin{proof}
We quickly review some elementary properties of $q_{N}(G^{u})$ from \cite{A1,
A2, A, EMS}. This polynomial is described recursively as follows. If $G$ has a
loop at $a$ then $q_{N}(G^{u})$ = $q_{N}((G^{u})^{a}-a)$ + $q_{N}(G^{u}-a)$,
if $a$ and $b$ are loopless neighbors then $q_{N}(G^{u})$ = $q_{N}(G^{u}-a)$ +
$q_{N}((G^{u})^{ab}-b)$, and if $E_{n}$ is the edgeless $n $-vertex graph then
$q_{N}(E_{n}^{u})=y^{n}$. The latter includes the empty graph $E_{0}$ with
$q_{N}(E_{0}^{u})=1$. It follows by induction on the number of vertices that
no graph has any negative coefficient in $q_{N}(G^{u})$, and hence every graph
has $\gamma(G)\geq0$ and $\varepsilon(G)\geq0$. Moreover, every nonempty
simple graph has $\varepsilon(G)=0$ and every disconnected simple graph has
$\gamma(G)=0$.

Theorem \ref{theorem5} is certainly true for $E_{n}$, as $\frac{1}{2}%
\gamma(E_{n}),\frac{1}{2}\varepsilon(E_{n})\leq1$.

Proceeding inductively, observe that if Corollary \ref{cor1} or part (a) or
(b) of Theorem \ref{theorem1} is applied to a graph $H$ represented by a
certain node of the computation tree, and $H_{1}$ and $H_{2}$ are the graphs
represented by the (first two) resulting branch nodes, then $q_{N}(H^{u})$ =
$q_{N}(H_{1}^{u})$ + $q_{N}(H_{2}^{u})$; certainly then $\frac{1}%
{2}\varepsilon(H)$ = $\frac{1}{2}\varepsilon(H_{1})$ + $\frac{1}{2}%
\varepsilon(H_{2})$. (By the way, we call vertices of the computation tree
\textit{nodes} in order to distinguish them from vertices of $G$.) If Theorem
\ref{theorem3} or \ref{theorem4} is applied to remove a vertex $b$ from a
graph $H $ then $\frac{1}{2}q_{N}(H^{u})(0)$ = $\frac{1}{2}q_{N}(H^{u}-b)(0)$,
because the formulas of Theorems \ref{theorem3} and \ref{theorem4} yield
$\alpha^{\prime}(a)=\beta^{\prime}(a)=1$ when $\alpha(a)$ = $\beta(a)$ =
$\alpha(b)$ = $\beta(b)$ = $1$, $x=2$ and $y=0$.

It remains to consider the special case involving simple graphs. We actually
prove a slightly different result, namely: if $G$ is simple then the portion
of the computation tree involving only nodes corresponding to connected graphs
has no fewer than $\frac{1}{2}\gamma(G)$ leaves. If $G$ is disconnected then
$\gamma(G)=0$ so $G$ satisfies the result trivially. If $G $ is a connected,
simple graph with $n\leq2$ then $G$ satisfies the result because $\frac{1}%
{2}\gamma(G)\leq1$. Proceeding inductively, observe that if Corollary
\ref{cor1} or Theorem \ref{theorem1} (b) is applied to a connected graph $H$
represented by a certain node of the computation tree, and $H_{1}$ and $H_{2}$
are the graphs represented by the (first two) resulting branch nodes, then
$q_{N}(H^{u})$ = $q_{N}(H_{1}^{u})$ + $q_{N}(H_{2}^{u})$. It follows that
$\frac{1}{2}\gamma(H)$ = $\frac{1}{2}\gamma(H_{1})$ + $\frac{1}{2}\gamma
(H_{2})$. If Theorem \ref{theorem3} or Theorem \ref{theorem4} is used to
remove a vertex $b$ from a connected, simple graph $H$ with 3 or more vertices
then $\gamma(H)=\gamma(H-b)$, by Corollary 4.17 of \cite{EMS}.
\end{proof}

\bigskip

Theorem \ref{theorem5} is of limited value because the computations discussed
are not optimal. The restriction that part (c) of Theorem \ref{theorem1} is
only applied to loopless graphs is an obvious inefficiency. In addition, if
some (combinations of) weights are 0 then it would be natural to simply ignore
the corresponding parts of the computation. There are also other useful twin
reductions that do not fall under Theorem \ref{theorem5}.

\begin{definition}
Two vertices $a,b$ of $G$ are\emph{\ fraternal twins }if (i) either they are
looped and nonadjacent or they are unlooped and adjacent, and (ii) they have
the same neighbors outside $\{a,b\}$.
\end{definition}

Unlooped fraternal twins are called \textit{true twins} in \cite{EMS}, but we
prefer the present terminology because the rows and columns of the adjacency
matrix corresponding to fraternal twins are not quite the same. Here is an
extension of Proposition 4.15 of \cite{EMS} to weighted graphs.

\begin{theorem}
\label{theorem6}Suppose $a$ and $b$ are fraternal twins in $G$. Let
$G^{\prime}$ be the graph obtained from $G-b$ by changing the weights of $a$:
$\alpha^{\prime}(a)$ = $\alpha(a)\beta(b)+\beta(a)\alpha(b)$ and
$\beta^{\prime}(a)$ = $\beta(a)\beta(b)+\alpha(a)\alpha(b)(x-1)^{2}$. Then
$q(G)=q(G^{\prime})$.
\end{theorem}

\begin{proof}
If $a\in S\subseteq V(G^{\prime})$ then $r(G^{\prime}[S])$ = $r(G[S])$\ =
$r(G[(S\cup\{b\})-\{a\}])$, because the\ adjacency matrices are the same. If
$a\not \in S\subseteq V(G^{\prime})$ then $r(G^{\prime}[S])$ = $r(G[S])$ =
$r(G[S\cup\{a,b\}])-2$.
\end{proof}

Theorem \ref{theorem6} has the following inductive generalization. Suppose
$k\geq1$ and $a=b_{0},b_{1},...,b_{k}$ are fraternal twins in $G$. If
$G^{\prime}$ is the graph obtained from $G-b_{1}-...-b_{k}$ by changing the
weights of $a$ to
\begin{align*}
\alpha^{\prime}(a)  &  =\sum_{\substack{S\subseteq\{b_{0},...,b_{k}\}
\\\left|  S\right|  ~\mathrm{odd}}}(%
%TCIMACRO{\dprod \limits_{b_{i}\in S}}%
%BeginExpansion
{\displaystyle\prod\limits_{b_{i}\in S}}
%EndExpansion
\alpha(b_{i}))(%
%TCIMACRO{\dprod \limits_{b_{j}\not \in S}}%
%BeginExpansion
{\displaystyle\prod\limits_{b_{j}\not \in S}}
%EndExpansion
\beta(b_{j}))(x-1)^{\left|  S\right|  -1}\\
& \\
\mathrm{and}~\beta^{\prime}(a)  &  =\sum_{\substack{S\subseteq\{b_{0}%
,...,b_{k}\} \\\left|  S\right|  ~\mathrm{even}}}(%
%TCIMACRO{\dprod \limits_{b_{i}\in S}}%
%BeginExpansion
{\displaystyle\prod\limits_{b_{i}\in S}}
%EndExpansion
\alpha(b_{i}))(%
%TCIMACRO{\dprod \limits_{b_{j}\not \in S}}%
%BeginExpansion
{\displaystyle\prod\limits_{b_{j}\not \in S}}
%EndExpansion
\beta(b_{j}))(x-1)^{\left|  S\right|  },
\end{align*}
\noindent then $q(G)=q(G^{\prime})$. We refer to the process of
combining\ several fraternal twins into a single vertex as a \textit{fraternal
twin reduction} no matter how many\ fraternal twins are combined.

\bigskip

In general, fraternal twin reductions are just as useful as identical twin
reductions. However, as noted in Proposition 38 of \cite{A2} and Corollary
4.16 of \cite{EMS} they have the effect of multiplying the unweighted
vertex-nullity polynomial by powers of 2. (Observe that if $\alpha(a)$ =
$\beta(a)$ = $\alpha(b)$ = $\beta(b)$ = $1$ and $x=2$ then Theorem
\ref{theorem6} gives $\alpha^{\prime}(a)$ = $\beta^{\prime}(a)$ =
$2$.)\ Consequently the lower bounds of Theorem \ref{theorem5} are not valid
for computations that utilize Theorem \ref{theorem6} along with Corollary
\ref{cor1} and Theorems \ref{theorem1}, \ref{theorem3} and \ref{theorem4}.

\bigskip

If we are given a reduction of a graph $G$ to disconnected vertices using twin
reductions and unlooped pendant vertex reductions, then Theorems
\ref{theorem3}, \ref{theorem4} and \ref{theorem6} describe $q(G)$ in linear
time -- simply update the vertex weights at each step, and at the end refer to
part (c) of Theorem \ref{theorem1}. If we are not given such a reduction then
determining whether or not any such reduction exists, and finding one if
possible, can be accomplished in polynomial time: as in Corollary 5.3 of
\cite{EMS}, simply search $V(G)$ repeatedly for unlooped degree-one vertices
and pairs of vertices $a,b$ with the same neighbors outside $\{a,b\}$. Hence
if $G$ has such a reduction then in polynomial time, Theorems \ref{theorem3},
\ref{theorem4} and \ref{theorem6} provide a description of $q(G)$ that
completely avoids the branching formulas of Corollary \ref{cor1} and parts (a)
and (b) of Theorem \ref{theorem1}. This observation extends Theorem 6.4 of
\cite{EMS} from $q_{N}$ to $q$ and from simple graphs that can be analyzed
without fraternal twin reductions to looped graphs that can be analyzed using
all three types of reductions:

\begin{theorem}
If a graph $G$ can be reduced to a collection of disconnected vertices using
unlooped pendant vertex reductions and the two types of twin reductions then
Theorems \ref{theorem3}, \ref{theorem4} and \ref{theorem6} provide a
polynomial-time description of $q(G)$.
\end{theorem}

The theorem refers to a \textit{description} rather than a
\textit{computation} because we have ignored the cost of arithmetic operations
in the ring $R$. In case $R=\mathbb{Q}$ arithmetic operations have low cost,
and describing a weighted \textquotedblleft polynomial\textquotedblright%
\ $q(G)$ is the same as computing an evaluation of the unweighted polynomial
$q(G)$. The full unweighted polynomial may be recovered from several
evaluations by interpolation, so the theorem provides a polynomial-time
computation of the unweighted polynomial. In more complicated rings like
$\mathbb{Z}[x,y,\alpha_{1},..,\alpha_{n},\beta_{1},...,\beta_{n}]$ arithmetic
operations may be so expensive as to prohibit polynomial-time computation of
entire weighted polynomials. A thorough discussion of these matters is given
by Courcelle \cite{Ci}.

\section{Composition}

In this section we reformulate and extend some results of Arratia,
Bollob\'{a}s and Sorkin \cite{A2} regarding substituted graphs, using the
following version of a construction introduced by Cunningham \cite{Cu}.

\begin{definition}
\label{comp}A vertex-weighted graph $G$ is the \emph{composition} of
vertex-weighted graphs $H$ and $K$, $G=H\ast K$, if the following conditions hold.

(a) $V(H)\cap V(K)$ consists of a single unlooped vertex $a$.

(b) The vertex $a$ is unweighted in both $H$ and $K$, i.e., $\alpha
(a)=1=\beta(a)$ in $H$ and $K$.

(c) $V(G)=(V(H)\cup V(K))-a$, and the vertices of $G$ inherit their weights
from $H$ and $K$.

(d) $E(G)=E(H)\cup E(K)\cup\{vw|va\in E(G)$ and $aw\in E(H)\}$.
\end{definition}

Requiring $a$ to be unlooped and unweighted in both $H$ and $K$ guarantees
that no information is lost when we remove $a$ in constructing $G$.%

%TCIMACRO{\FRAME{fhFU}{4.5956in}{1.3915in}{0pt}{\Qcb{Composition of graphs.}%
%}{\Qlb{weighif}}{weighif.ps}{\special{ language "Scientific Word";
%type "GRAPHIC";  maintain-aspect-ratio TRUE;  display "USEDEF";
%valid_file "F";  width 4.5956in;  height 1.3915in;  depth 0pt;
%original-width 8.246in;  original-height 10.6969in;  cropleft "0.1390";
%croptop "0.9192";  cropright "0.9304";  cropbottom "0.7371";
%filename 'weighif.ps';file-properties "XNPEU";}}}%
%BeginExpansion
\begin{figure}
[h]
\begin{center}
\includegraphics[
trim=1.146194in 7.884686in 0.573922in 0.864309in,
height=1.3915in,
width=4.5956in
]%
{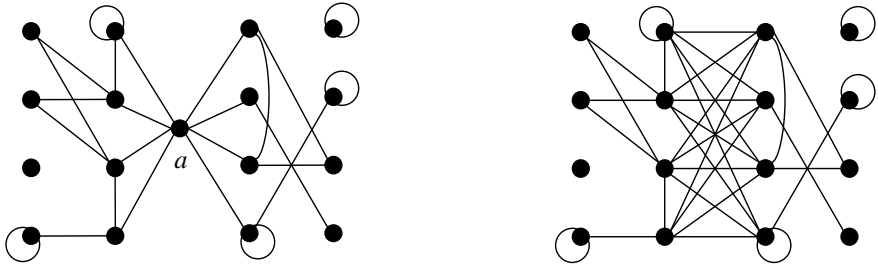}%
\caption{Composition of graphs.}%
\label{weighif}%
\end{center}
\end{figure}
%EndExpansion

Definition \ref{comp} includes several other familiar notions. If $a$ is
isolated in $H$ or $K$ then $H\ast K$ is the disjoint union of $H-a$ and $K-a
$. If $a$ is adjacent to every other vertex of $H$ and $K$ then $H\ast K$ is
also denoted $(H-a)+(K-a)$. This is traditionally called a \textquotedblleft
join\textquotedblright\ but that term has recently been used for general
compositions \cite{Cj}. If $a$ is adjacent to every other vertex of $H$ then
$H\ast K$ is the graph obtained by substituting $H-a$ for $a$ in $K$. If $a$
is adjacent to every other vertex of $H$ and $H-a$ is edgeless or complete
then the (un)looped vertices of $H-a$ are twins in $H\ast K$.

\bigskip

The following observation will be useful.

\begin{proposition}
\label{prop2} Suppose two graphs $\Gamma_{1}$ and $\Gamma_{2}$ are identical
except for the weights of a single vertex $a$, and let $\Gamma$ be the graph
that is identical to both $\Gamma_{1}$ and $\Gamma_{2}$ except for
$\alpha_{\Gamma}(a)=\alpha_{\Gamma_{1}}(a)+\alpha_{\Gamma_{2}}(a)$ and
$\beta_{\Gamma}(a)=\beta_{\Gamma_{1}}(a)+\beta_{\Gamma_{2}}(a)$. Then
$q(\Gamma_{1})+q(\Gamma_{2})=q(\Gamma)$.
\end{proposition}

Let $a$ be an unweighted vertex of a simple vertex-weighted graph $H$, and let
$H_{1}$ and $H_{2}$ be the full subgraphs of $H-a$ induced by the neighbors
(resp. non-neighbors) of $a$. Any weighted interlace polynomial $q(H\ast K)$
may be analyzed in the following way.

\bigskip

Step 1. Eliminate all edges between vertices of $H-a$ using pivots and
weight-changes as in\ Corollary \ref{cor1}. (In general there will be many
different sequences of pivots that may be used; this lack of uniqueness is not
important in the analysis.) The assignments of individual vertices of $H$ to
$H_{1}$ and $H_{2}$ may change during this process, and vertex weights may
also change; but these reassignments and weight changes will be the same for
all graphs $K$. As in the proof of Proposition 39 of \cite{A2} these pivots
will not affect the internal structure of $K$, because no two vertices of $K$
have distinct, nonempty sets of neighbors in $H$. Here the phrase ``internal
structure''\ refers to vertex weights, the positions of loops, and
adjacencies, including adjacencies between $a$ and other vertices of $K$.

\bigskip

Step 2. When Step 1 is complete, $q(H\ast K)$ is expressed as a sum in which
each summand is the product of an initial multiplying factor and the weighted
interlace polynomial $q(H^{\prime}\ast K)$ of a graph in which every edge is
incident on a vertex of $K-a$. If $v\in V(H_{2}^{\prime})$ then $v$ is
isolated, so the only effect of $v$ is to multiply that summand by
$q(\{v\})=\alpha(v)(y-1)+\beta(v)$. This same effect is realized by removing
$v$ and incorporating $\alpha(v)(y-1)+\beta(v)$ into the initial multiplying
factor, so we may assume that every summand has $H_{2}^{\prime}=\emptyset$. In
a summand with $\left\vert V(H_{1}^{\prime})\right\vert >1 $ the vertices of
$H_{1}^{\prime}$ are all identical twins, and may be consolidated into a
single vertex $a$ using Theorem \ref{theorem3}. In a summand with $\left\vert
V(H_{1}^{\prime})\right\vert =0$ a single vertex $a\in V(H_{1}^{\prime})$ may
be introduced with $\alpha(a)=0$ and $\beta(a)=1$; this will not affect the
value of the corresponding summand. As in Step 1, these manipulations are the
same for all $K$.

\bigskip

Step 3. The weighted interlace polynomial $q(H\ast K)$ is now expressed as a
sum in which each summand is the product of an initial multiplying factor and
a weighted interlace polynomial $q(K^{\prime})$ in which $K^{\prime}$ differs
from $K$ only in the weights of $a$. In each summand we multiply the $\alpha$
and $\beta$ weights of $a$ by the initial multiplying factor. This has the
effect of multiplying $q(K^{\prime})$ by that factor, so the summand is now
simply $q(K^{\prime})$. Proposition \ref{prop2} tells us that the sum is equal
to a single weighted interlace polynomial $q(K^{\prime})$, where $\alpha(a)$
and $\beta(a)$ are obtained by adding together the $\alpha$ and $\beta$
weights of $a$ in the various summands.

\bigskip

We deduce that there are weights $\alpha(a)$ and $\beta(a)$ that depend only
on $H$ and $a$, and have the following property: In every instance of
Definition \ref{comp} involving $H$, the interlace polynomial $q(H\ast K)$
equals $q(K^{\prime})$, where $K^{\prime}$ is obtained from \thinspace$K$ by
using $\alpha(a)$ and $\beta(a)$ as weights for $a$. Finding explicit formulas
for these weights is not difficult.

\begin{theorem}
\label{theorem8}Let $H$ be a vertex-weighted simple graph with an unweighted
vertex $a$. Then every composition $H\ast K$ has $q(H\ast K)=q(K^{\prime})$,
where $K^{\prime}$ is obtained from \thinspace$K$ by using the following
weights for $a$.
\begin{align*}
\alpha(a)  &  =\frac{q(H)-yq(H-a)}{(x-1)^{2}-(y-1)^{2}}\\
&  ~\\
\beta(a)  &  =\frac{((x-1)^{2}+y-1)q(H-a)-(y-1)q(H)}{(x-1)^{2}-(y-1)^{2}}%
\end{align*}

\end{theorem}

\begin{proof}
With $V(K)=\{a\}$, we have $q(H-a)$ = $q(H\ast K)$ = $q(K^{\prime})$ =
$\beta(a)+(y-1)\alpha(a)$. With $K$ consisting of two adjacent, unlooped,
unweighted vertices $a$ and $v$ we have $q(H)$ = $q(H\ast K)$ = $q(K^{\prime
})$\ = $((x-1)^{2}+y-1)\alpha(a)+y\beta(a)$. The stated formulas for
$\alpha(a)$ and $\beta(a)$ follow.
\end{proof}

\bigskip

In case $(x-1)^{2}-(y-1)^{2}$ might be a zero divisor in the ring $R$, one can
avoid any difficulty with the formulas of Theorem \ref{theorem8} by first
evaluating them in the polynomial ring $\mathbb{Z}[x,y,\alpha_{1}%
,..,\alpha_{n},\beta_{1},...,\beta_{n}]$, and then evaluating the resulting
division-free formulas in $R$.

\bigskip

Theorem \ref{theorem8} concerns compositions $H\ast K$ in which $H$ is simple.
If $H$ has looped vertices, a similar analysis requires two new steps.

\bigskip

Step 0. Begin by removing all loops in $H-a$ using local complementation as in
part (a) of Theorem \ref{theorem1}. The result is a description of $q(H\ast
K)$ as a sum in which each summand is the product of an initial factor and an
interlace polynomial $q(H^{\prime}\ast K)$ or $q(H^{\prime}\ast K^{a})$, with
no loops in $H^{\prime}$.

\bigskip

Step 4. After applying steps 1 - 3 to each of these summands, collect terms to
obtain a formula
\[
q(H\ast K)=q(K^{\prime})+q((K^{a})^{\prime})\text{.}%
\]
In order to distinguish the two terms on the right-hand side we denote by
$a_{c}$ the copy of $a$ in $K^{a}$. It might seem that we now have to
determine four unknowns, namely the vertex weights $\alpha(a)$, $\alpha
(a_{c})$, $\beta(a)$ and $\beta(a_{c})$ It turns out though that these
unknowns are not independent. There is an obvious isomorphism between
$(K^{\prime})^{a}$ and $K^{a}$, and consequently
\[
q((K^{\prime})^{a})-\beta(a)q((K^{\prime})^{a}-a)
\]
may be obtained from
\[
q((K^{a})^{\prime})-\beta(a_{c})q((K^{a})^{\prime}-a_{c})
\]
simply by replacing $\alpha(a_{c})$ with $\alpha(a)$. Theorem \ref{theorem2}
tells us that
\[
q(K^{\prime})-\beta(a)q(K^{\prime}-a)=q((K^{\prime})^{a})-\beta(a)q((K^{\prime
})^{a}-a),
\]
so
\[
q(K^{\prime})-\beta(a)q(K^{\prime}-a)
\]
may be obtained from
\[
q((K^{a})^{\prime})-\beta(a_{c})q((K^{a})^{\prime}-a_{c})
\]
by replacing $\alpha(a_{c})$ with $\alpha(a)$. That is, the coefficient of
$\alpha(a)$ in $q(K^{\prime})$ is precisely the same as the coefficient of
$\alpha(a_{c})$ in $q((K^{a})^{\prime})$.\ It follows that the sum
$q(K^{\prime})$ + $q((K^{a})^{\prime})$ is unchanged if we replace $\alpha(a)$
by $\alpha(a)$ + $\alpha(a_{c})$ and replace $\alpha(a_{c})$ by $0$.

\begin{theorem}
\label{theorem9}Let $H$ be a vertex-weighted graph with an unweighted,
unlooped vertex $a$. Then $H$ and $a$ determine weights $\alpha(a)$,
$\beta(a)$ and $\beta(a_{c})$ such that every composition $H\ast K$ has
$q(H\ast K)=q(K^{\prime})+q((K^{a})^{\prime})$, where $K^{\prime}$ is obtained
from \thinspace$K$ by using $\alpha(a)$ and $\beta(a)$ as weights for $a$ and
$(K^{a})^{\prime}$ is obtained from \thinspace$K^{a}$ by using $\alpha
(a_{c})=0$ and $\beta(a_{c})$ as weights for $a_{c}$, the copy of $a$ in
$(K^{a})^{\prime}$.
\end{theorem}

Formulas for the three weights mentioned in\ Theorem \ref{theorem9} may be
derived from three instances of the theorem. We use $H-a$ and $H$ as in
Theorem \ref{theorem8}, and also $H^{\ell}$, the graph obtained from $H$ by
attaching a loop at $a$. These correspond respectively to compositions of $H$
with graphs $K_{1},K_{2},K_{3}$ such that $V(K_{1})=\{a\}$; $V(K_{2})=\{a,v\}
$ with $v$ an unweighted, unlooped neighbor of $a$; and $V(K_{3})=\{a,\ell\} $
with $\ell$ an unweighted, looped neighbor of $a$. Definition \ref{2varq}
gives the following values.
\[
q(H-a)=q(K_{1}^{\prime})+q((K_{1}^{a})^{\prime})=(y-1)\alpha(a)+\beta
(a)+\beta(a_{c})
\]

\[
q(H)=q(K_{2}^{\prime})+q((K_{2}^{a})^{\prime})=((x-1)^{2}+y-1)\alpha
(a)+y\beta(a)+x\beta(a_{c})
\]

\[
q(H^{\ell})=q(K_{3}^{\prime})+q((K_{3}^{a})^{\prime})=((x-1)^{2}%
+y-1)\alpha(a)+x\beta(a)+y\beta(a_{c})
\]
We deduce these formulas.
\[
\alpha(a)=\frac{(x+y)q(H-a)-q(H)-q(H^{\ell})}{(x+y)(y-1)-2((x-1)^{2}+y-1)}%
\]

\[
\beta(a)-\beta(a_{c})=\frac{q(H)-q(H^{\ell})}{y-x}%
\]

\[
\beta(a)+\beta(a_{c})=\frac{2((x-1)^{2}+y-1)q(H-a)-(y-1)\left(  q(H)+q(H^{\ell
})\right)  }{2((x-1)^{2}+y-1)-(y-1)(x+y)}%
\]
Separate formulas for $\beta(a)$ and $\beta(a_{c})$ are derived in the obvious
ways by adding and subtracting the last two. As before, possible problems with
denominators may be avoided by first evaluating the formulas in $\mathbb{Z}%
[x,y,\alpha_{1},..,\alpha_{n},\beta_{1},...,\beta_{n}]$.

\bigskip

We close this section with different formulas for the weights $\alpha(a)$,
$\beta(a)$ and $\beta(a_{c})$ that appear in Theorems \ref{theorem8} and
\ref{theorem9}. Suppose $H$ has an unweighted, unlooped vertex $a$. Let $N(a)
$ denote the open neighborhood of $a$, i.e., the set containing the vertices
$v\neq a\in V(H)$ that neighbor $a$. Given a\ subset $S\subseteq V(H-a)$ let
$\rho=\rho_{S,a}$ (resp. $\kappa=\kappa_{S,a}$) be the row (resp. column)
vector with entries indexed by $\{i:v_{i}\in S\}$ whose $i^{th}$ entry is 1 or
0 according to whether $v_{i}\in N(A)$ or $v_{i}\not \in N(a)$. Also let
$M=M_{S}$ be the adjacency matrix of $H[S]$. Note that $r(M)\leq r
\begin{pmatrix}
M & \kappa
\end{pmatrix}
\leq r(M)+1$ and
\[
r(M)\leq r\left(
\begin{array}
[c]{cc}%
M & \kappa\\
\rho & 0
\end{array}
\right)  ,\;\;r\left(
\begin{array}
[c]{cc}%
M & \kappa\\
\rho & 1
\end{array}
\right)  \leq r(M)+2
\]
because adjoining a single row or column to a matrix raises the rank by 0 or 1.

\begin{definition}
\label{def9}The \emph{type} of $S$ with respect to $a$ is defined as follows.

\bigskip

\noindent$S$ is of type 1 if $r(M)=r\left(
\begin{array}
[c]{cc}%
M & \kappa\\
\rho & 0
\end{array}
\right)  =r\left(
\begin{array}
[c]{cc}%
M & \kappa\\
\rho & 1
\end{array}
\right)  -1.$

\bigskip

\noindent$S$ is of type 2 if $r(M)+2=r\left(
\begin{array}
[c]{cc}%
M & \kappa\\
\rho & 0
\end{array}
\right)  =r\left(
\begin{array}
[c]{cc}%
M & \kappa\\
\rho & 1
\end{array}
\right)  .$

\bigskip

\noindent$S$ is of type 3 if $r(M)=r\left(
\begin{array}
[c]{cc}%
M & \kappa\\
\rho & 1
\end{array}
\right)  =r\left(
\begin{array}
[c]{cc}%
M & \kappa\\
\rho & 0
\end{array}
\right)  -1.$
\end{definition}

\begin{lemma}
Every $S\subseteq V(H-a)$ is of one of these types. Moreover, if $a$ has no
looped neighbor then there is no $S\subseteq V(H-a)$ of type 3.
\end{lemma}

\begin{proof}
The fact that every $S\subseteq V(H-a)$ is of type 1, 2 or 3 appears in Lemma
2 of \cite{BBCS}.

Suppose $S$ is of type 3, so
\[
r(M)+1=r\left(
\begin{array}
[c]{cc}%
M & \kappa\\
\rho & 0
\end{array}
\right)  .
\]
The row vector $%
\begin{pmatrix}
\rho & 0
\end{pmatrix}
$ cannot be a sum of rows of $%
\begin{pmatrix}
M & \kappa
\end{pmatrix}
$. For if it were then $\rho$ would be the corresponding sum of rows of $M$,
and by symmetry $\kappa$ would be the corresponding sum of columns of $M$.
Consequently it would follow that
\[
r(M)=r\left(
\begin{array}
[c]{cc}%
M & \kappa
\end{array}
\right)  =r\left(
\begin{array}
[c]{cc}%
M & \kappa\\
\rho & 0
\end{array}
\right)  .
\]
On the other hand, $\rho$ must be the sum of the rows of $M$ corresponding to
the elements of some subset $T\subseteq S$, for if it were not then it would
follow that
\[
r(M)+1=r\left(
\begin{array}
[c]{c}%
M\\
\rho
\end{array}
\right)  =r\left(
\begin{array}
[c]{cc}%
M & \kappa
\end{array}
\right)  =r\left(
\begin{array}
[c]{cc}%
M & \kappa\\
\rho & 0
\end{array}
\right)  -1.
\]
Every such $T$ must contain an odd number of neighbors of $a$, to avoid giving
a sum of rows of $%
\begin{pmatrix}
M & \kappa
\end{pmatrix}
$ equal to $%
\begin{pmatrix}
\rho & 0
\end{pmatrix}
$.

Choose such a $T$, and consider the induced subgraph $H[T\cap N(a)]$. As the
sum of the rows of $M$ corresponding to elements of $T$ is $\rho$, the sum of
the rows of the adjacency matrix of $H[T\cap N(a)]$ is the row vector
$(1...1)$. $\left\vert T\cap N(a)\right\vert $ is odd, so the adjacency matrix
of $H[T\cap N(a)]$\ has an odd number of entries equal to 1. The matrix is
symmetric, so at least one of these entries must appear on the diagonal. That
is, at least one vertex of $H[T\cap N(a)]$ is looped.
\end{proof}

\bigskip

For $S\subseteq V(H-a)$ let $M^{a}=M_{S}^{a}$ be the matrix obtained from $M$
by toggling every entry $m_{ij}$ that has $v_{i},v_{j}\in N(a)$. Then $r
\begin{pmatrix}
M & \kappa
\end{pmatrix}
$ = $r
\begin{pmatrix}
M^{a} & \kappa
\end{pmatrix}
$ because the first matrix is transformed into the second by adding the last
column to every column corresponding to a neighbor of $a$. Similarly, adding
the last column of the first matrix below to every column corresponding to a
neighbor of $a$ tells us that
\[
r\left(
\begin{array}
[c]{cc}%
M & \kappa\\
\rho & 1
\end{array}
\right)  =r\left(
\begin{array}
[c]{cc}%
M^{a} & \kappa\\
0 & 1
\end{array}
\right)  =1+r(M^{a}).
\]
Consequently, Definition \ref{def9} may be restated using the relationship
between $r(M)$ and $r(M^{a})$: if $S$ is of type 1 then $r(M)=r(M^{a})$, if
$S$ is of type 2 then $r(M)=r(M^{a})-1$, and if $S$ is of type 3 then
$r(M)=r(M^{a})+1$.

\begin{proposition}
\label{prop3}For $i\in\{1,2,3\}$ let
\[
q_{i}(H-a)=\sum_{\substack{S\subseteq V(H-a) \\\text{of type }i}}(%
%TCIMACRO{\dprod \limits_{s\in S}}%
%BeginExpansion
{\displaystyle\prod\limits_{s\in S}}
%EndExpansion
\alpha(s))(%
%TCIMACRO{\dprod \limits_{v\not \in S}}%
%BeginExpansion
{\displaystyle\prod\limits_{v\not \in S}}
%EndExpansion
\beta(v))(x-1)^{r((H-a)[S])}(y-1)^{n((H-a)[S])}.
\]
Then the following equations hold.
\begin{align*}
q(H-a)  &  =q_{1}(H-a)+q_{2}(H-a)+q_{3}(H-a)\\
&  \,\\
q(H)  &  =yq_{1}(H-a)+\left(  1+\frac{(x-1)^{2}}{y-1}\right)  q_{2}%
(H-a)+xq_{3}(H-a)\\
&  \,\\
q(H^{\ell})  &  =xq_{1}(H-a)+\left(  1+\frac{(x-1)^{2}}{y-1}\right)
q_{2}(H-a)+yq_{3}(H-a)
\end{align*}
Also,
\[
q(H^{a}-a)=q_{1}(H-a)+\left(  \frac{x-1}{y-1}\right)  q_{2}(H-a)+\left(
\frac{y-1}{x-1}\right)  q_{3}(H-a).
\]

\end{proposition}

\begin{proof}
The first equality is obvious. For the second, note that each $S\subseteq
V(H-a)$ gives rise to two subsets of $V(H)$, namely $S$ and $S\cup\{a\}$;
these correspond to the adjacency matrices $M$ and $\left(
\begin{array}
[c]{cc}%
M & \kappa\\
\rho & 0
\end{array}
\right)  $. Similarly, the third equality is derived by considering the
contributions of two adjacency matrices for each $S\subseteq V(H-a)$, namely
$M$ and $\left(
\begin{array}
[c]{cc}%
M & \kappa\\
\rho & 1
\end{array}
\right)  $. The last equality follows from the discussion preceding the proposition.
\end{proof}

\begin{corollary}
The weights mentioned in Theorems \ref{theorem8} and \ref{theorem9} are
$\beta(a)=q_{1}(H-a)$, $\alpha(a)=q_{2}(H-a)/(y-1)$ and $\beta(a_{c}%
)=q_{3}(H-a)$. In particular, $\beta(a_{c})=0$ if $a$ has no looped neighbor.
\end{corollary}

\begin{proof}
The corollary follows from Proposition \ref{prop3} and the formulas given
in\ Theorem \ref{theorem8} and immediately after Theorem \ref{theorem9}.
\end{proof}

\section{A characterization of simple graphs}

In this section we focus our attention on the unweighted vertex-nullity polynomial.

\begin{proposition}
\label{prop4}If $G$ is a connected, unweighted graph with at least one looped
vertex then $\varepsilon(G)=q_{N}(G)(0)>1$.
\end{proposition}

\begin{proof}
The proposition is certainly true if $G$ has $n\leq2$ vertices, as all three
such graphs have $\varepsilon(G)=2$. The argument proceeds by induction on
$n\geq3$. Recall that $\varepsilon(G)\geq0$ for every graph $G$, and let $a$
be a looped vertex of $G$. Then $q_{N}(G)=q_{N}(G-a)+q_{N}(G^{a}-a)$, so
$\varepsilon(G)\geq\max\{\varepsilon(G-a)$, $\varepsilon(G^{a}-a)\}$.

Suppose $a$ is not a cutpoint of $G$. If $G$ has some looped vertex other than
$a$ then the inductive hypothesis implies that $\varepsilon(G-a)>1$. If $G$
has no looped vertex other than $a$ then every neighbor of $a$ is looped in
$G^{a}-a$. Every component of $G^{a}-a$ contains at least one neighbor of $a$,
so the inductive hypothesis implies that $\varepsilon(C)>1$ for every
component $C$ of $G^{a}-a$. Hence $\varepsilon(G^{a}-a)=\prod_{C}%
\varepsilon(C)>1$.

Suppose now that $G$ has a looped cutpoint $a$. For each component $C$ of
$G-a$ and each vertex $v\in N(a)$ that lies in some other component of $G-a$,
no edge connecting $v$ to an element of $N(a)\cap V(C)$ appears in $G-a$.
Consequently every such edge appears in $G^{a}-a$, so all of $N(a)$ is
contained in a single component of $G^{a}-a$. As $G$ is connected, this
implies that $G^{a}-a$ is also connected. Hence if any neighbor of $a$ is
unlooped in $G$, $G^{a}-a$ is a connected graph with a looped vertex and the
inductive hypothesis implies that $\varepsilon(G^{a}-a)>1$. If instead every
neighbor of $a$ is looped in $G$, then every component $C$ of $G-a$ has a
looped vertex and the inductive hypothesis implies $\varepsilon(G-a)=\prod
_{C}\varepsilon(C)>1$.
\end{proof}

\begin{corollary}
\label{cor4}An unweighted graph $G$ has $\varepsilon(G)>0$ if and only if $G $
has no nonempty simple component.
\end{corollary}

\begin{proof}
As noted in Remark 20 of \cite{A2}, every nonempty simple graph has
$\varepsilon(G)=0$. It follows that every graph with a nonempty simple
component also has $\varepsilon(G)=0$. On the other hand, Proposition
\ref{prop4} and the fact that $\varepsilon(E_{0})=1$ together imply that every
graph with no nonempty simple component has $\varepsilon(G)>0$.
\end{proof}

\begin{corollary}
Let $G$ be an unweighted graph, and let $G^{c}$ be its \emph{complement},
i.e., the graph with $V(G^{c})=V(G)$ whose edges (including loops) are
precisely the edges absent in $G$. Then the following are equivalent: $G$ is
simple, $q_{N}(G+E_{1}^{u})=yq_{N}(G^{c})$ and $y|q_{N}(G+E_{1}^{u})$.
\end{corollary}

\begin{proof}
Suppose first that $G$ is simple, and let $H=G+E_{1}^{u}$ with $a$ the vertex
of $E_{1}^{u}$. Let $K$ be the two-vertex graph with $V(K)=\{a,\ell\} $ in
which $\ell$ is an unweighted, looped neighbor of $a$. Theorem \ref{theorem8}
tells us that
\[
q(H\ast K)=q(K^{\prime})=x\beta(a)+((x-1)^{2}+y-1)\alpha(a).
\]
On the other hand, the recursive description of $q$ \cite{A} tells us that
\[
q(H\ast K)=q((H\ast K)-\ell)+(x-1)q((H\ast K)^{\ell}-\ell)=q(G)+(x-1)q(G^{c}%
).
\]
Recalling that $q(G)=q(H-a)$ = $\beta(a)+(y-1)\alpha(a)$, we see that
\[
0=q(H\ast K)-q(H\ast K)=(x-1)\beta(a)+(x-1)^{2}\alpha(a)-(x-1)q(G^{c})
\]
and consequently $q(G^{c})=\beta(a)+(x-1)\alpha(a)$. Recalling that
$q(G+E_{1}^{u})$ = $q(H)$ = $((x-1)^{2}+y-1)\alpha(a)+y\beta(a)$, we see that
\[
q_{N}(G+E_{1}^{u})=q(G+E_{1}^{u})|_{x=2}=y\left(  \alpha(a)+\beta(a)\right)
|_{x=2}=yq(G^{c})|_{x=2}=yq_{N}(G^{c}).
\]

Now suppose that $G$ is not simple. Then $G$ has a looped vertex, and so does
$G+E_{1}^{u}$. As $G+E_{1}^{u}$ is connected, Corollary \ref{cor4} states that
$\varepsilon(G+E_{1}^{u})>0$; consequently $y\nmid q_{N}(G+E_{1}^{u})$.
\end{proof}

\section{Trees}

A combinatorial description of the interlace polynomials of trees and forests
is given in \cite{ACRT}. In order to motivate the next section we sketch this
description briefly here, omitting details and proofs. Recall that a tree $T$
is \textit{rooted} by specifying a root vertex $r\in V(T)$. Each non-root
vertex $v\in V(T)$ then has a unique \textit{parent} $p(v)$, a neighbor whose
distance from $r$ is less than the distance from $r$ to $v$. The elements of
$p^{-1}(\{p(v)\})$ are the \textit{children} of $p(v)$, and the children of
$p(v)$ other than $v$ itself are \textit{siblings} of $v$. An \textit{ordered}
tree is a rooted tree given with an order on the set of children of each
parent vertex; non-root vertices may then have \textit{earlier siblings} and
\textit{later siblings}. A set of vertices that contains no adjacent pair is
\textit{independent}, and a set of vertices \textit{dominates} a vertex $v$ if
it contains $v$ or contains some neighbor of $v$.

\begin{definition}
An \emph{earlier sibling cover} (or \emph{es-cover}) in an ordered tree $T$ is
an independent set $I$ that dominates $r$ and has the property that for every
non-root vertex $v\in I$, every earlier sibling of $v$ is dominated by $I$.
\end{definition}

\begin{definition}
For integers $s$ and $t$ the \emph{es-number} $c_{s,t}(T)$ is the number of
$s$-element es-covers in $T$ whose non-root elements have $t$ different parents.
\end{definition}

If $T^{\prime}$ is a subtree of $T$ and $r\in V(T^{\prime})$ then we presume
that the children of each parent vertex $v$ in $T^{\prime}$ are ordered by
restricting the order of the children of $v$ in $T$. With this convention, it
is easy to verify that the earlier sibling covers in large trees arise from
earlier sibling covers in subtrees.

\begin{lemma}
\label{recur} Let $T$ be an ordered tree with a leaf $\ell$ such that
$p(\ell)\neq r\neq\ell$, all the siblings of $\ell$ are leaves, and $\ell$ has
no later siblings. Then
\[
\{\text{es-covers }I\text{ in }T\text{ with }\ell\not \in
I\}=\{\text{es-covers in }T-\ell\}
\]
and
\begin{gather*}
\{\text{es-covers }I\text{ in }T\text{ with }\ell\in I\}\\
=\{\text{unions }p^{-1}(\{p(\ell)\})\cup I\text{ with }I\text{ an es-cover in
}T-p(\ell)-p^{-1}(\{p(\ell)\})\}.
\end{gather*}

\end{lemma}

Lemma \ref{recur} is the key to an inductive proof that the terms in the
definition of $q(T)$ can be collected into sub-totals corresponding to earlier
sibling covers.

\begin{definition}
Let $T$ be an ordered tree with vertex weights, and let $I$ be an es-cover in
$T$. Let
\begin{align*}
I_{r}  &  =\{v\in I:\text{either v}=\text{r or I contains a later sibling of
v}\},\\
I_{l}  &  =\{v\in I:\text{v}\neq\text{r and I contains no later sibling of
v}\},\text{ and }\\
I_{c}^{\prime}  &  =\{v\not \in I:\text{I contains a child of v}\}
\end{align*}
For each vertex $v$ define the $I$\emph{-weight} $w_{I}(v)$ as follows:
\[
w_{I}(v)=
\begin{cases}
\beta(v)+\alpha(v)(y-1) & \text{if $v\in I_{r}$}\\
\alpha(v)\cdot\left(  (y-1)\beta(p(v))+(x-1)^{2}\alpha(p(v))\right)  &
\text{if $v\in I_{l}$}\\
1 & \text{if $v\in I_{c}^{\prime}$}\\
\beta(v) & \text{if }v\notin I\text{ and $v\notin I_{c}^{\prime}$.}%
\end{cases}
\]
The product
\[
\prod_{v\in V(T)}w_{I}(v)
\]
is the \emph{total weight} of $I$ in $T$, denoted $w_{T}(I)$.
\end{definition}

\begin{theorem}
\label{treepoly} If $T$ is an ordered tree with vertex weights then
\[
q(T)=\sum_{I~an~es-cover}w_{T}(I)\text{.}%
\]

\end{theorem}

\begin{proof}
If $T$ has no more than one vertex of degree $\geq2$, it is not difficult to
verify the theorem directly. If $T$ has more than one vertex of degree $\geq2$
then the theorem follows inductively from Theorem \ref{theorem1} and Corollary
\ref{cor1} using Lemma \ref{recur}, as detailed in \cite{ACRT}.
\end{proof}

Setting $\alpha\equiv1$ and $\beta\equiv1$ we deduce that the unweighted
interlace polynomial of a tree is determined by a simple formula involving
earlier sibling covers.

\begin{corollary}
\label{treeint}If $T$ is a tree then the unweighted interlace polynomial of
$T$ is
\[
\sum_{s,t}c_{s,t}(T)\cdot y^{s-t}(y-1+(x-1)^{2})^{t}\text{.}%
\]

\end{corollary}

As the interlace polynomials are multiplicative on disjoint unions, these
results extend directly to disconnected forests.

\section{Algorithmic activities}

The Tutte polynomial is a very useful invariant of graphs and matroids, which
incorporates a great deal of information and can be defined in several
different ways; see \cite{Bo}, \cite{T} and \cite{W} for more detailed
discussions than we provide here. One of its definitions is given in
Definitions \ref{oldact} and \ref{act}.

\begin{definition}
\label{oldact} Suppose $G$ is an unweighted graph with $E(G)=\{e_{1}%
,...,e_{m}\}$, and $T$ is a\ maximal spanning forest of $G$. An element
$e_{i}\not \in E(T)$ is \emph{externally active} with respect to $T$ if $i$ is
the least index of an element of the unique circuit contained in
$E(T)\cup\{e_{i}\}$. An element $e_{i}\in E(T)$ is \emph{internally active}
with respect to $T$ if $i$ is the least index of an element of the unique
cutset contained in $(E(G)-E(T))\cup\{e_{i}\}$. The numbers of edges that are
externally and internally active with respect to $T$ are denoted $e(T)$ and
$i(T)$, respectively.
\end{definition}

\begin{definition}
\label{act} If $G$ is an unweighted graph with $E(G)=\{e_{1},...,e_{m}\}$ then
the Tutte polynomial of $G$ is
\[
t(G)=\sum_{T}x^{i(T)}y^{e(T)}.
\]

\end{definition}

Theorem \ref{treepoly} and Corollary \ref{treeint} bear a strong resemblance
to this formula, with earlier sibling covers in rooted trees replacing maximal
spanning forests in arbitrary graphs. We do not know whether there is a
combinatorial analogue of Definition \ref{act} that describes the interlace
polynomials of an arbitrary graph, but there is an algorithmic analogue.
Before presenting it, we recall another definition of the Tutte polynomial.

\begin{definition}
\label{delcont} If $G$ is an unweighted graph then the Tutte polynomial $t(G)
$ is an element of the polynomial ring $\mathbb{Z}[x,y]$ determined
recursively by these properties.

(a) If $e\in E(G)$ is neither a loop nor an isthmus then $t(G)=t(G-e)+t(G/e)$.

(b) If $\lambda\in E(G)$ is a loop then $t(G)=yt(G-\lambda)$.

(c) If $\beta\in E(G)$ is an isthmus of $G$ then $t(G)=xt(G/\beta)$.

(d) For any positive integer $n$, $t(E_{n})=1$.
\end{definition}

To recursively calculate $t(G)$ one simply chooses an arbitrary edge of $G$,
and applies the appropriate part of Definition \ref{delcont}; this process is
repeated as many times as necessary. Such a computation is represented by a
computation tree in which a node that represents an instance of part (a) has
two children and a node that represents an instance of part (b) or (c) has
only one.

\begin{proposition}
\label{actact} Let $G$ be a graph with $E(G)=\{e_{1},...,e_{m}\}$, and
consider the recursive implementation of Definition \ref{delcont} in which
$e_{m}$ is removed first, then $e_{m-1}$ is removed in all branches, then
$e_{m-2}$ is removed in all branches, and so on. The leaves of the computation
tree representing this implementation correspond to the maximal spanning
forests of $G$, with the leaf corresponding to $T$ resulting from the portion
of the computation in which edges of $T$ are contracted and elements of
$E(G)-E(T)$ are deleted. A node of this portion of the computation tree
represents the removal of an active edge if and only if it has precisely one child.
\end{proposition}

\begin{proof}
The proposition is implicit in the fact that Definition \ref{delcont} and
Definition \ref{act} both yield $t(G)$, so it appears implicitly in just about
every presentation of the Tutte polynomial. See Theorem IX. 65 of \cite{T} or
Theorem X.10 of \cite{Bo}, for instance. Explicit discussions of the
connection between activities and computation are less common in the
literature, though there are some \cite{Ba, GM, GT}.

The proof is a direct induction on $\left|  E(G)\right|  $. If $e_{m}$ is a
loop then the maximal spanning forests of $G$ and $G-e_{m}$ coincide, and the
computation tree for $t(G)$ is obtained from the computation tree for
$t(G-e_{m})$ by attaching a new root node of degree 1, representing the
removal of $e_{m}$. If $e_{m}$ is an isthmus then the maximal spanning forests
of $G$ and $G/e_{m}$ correspond, and the computation tree for $t(G)$ is
obtained from the computation tree for $t(G/e_{m})$ by attaching a new root
node of degree 1. Otherwise, the maximal spanning forests of $G$ that contain
$e_{m}$ correspond to the maximal spanning forests of $G/e_{m}$, and the
maximal spanning forests of $G$ that do not contain $e_{m}$ are the maximal
spanning forests of $G-e_{m}$. The computation tree for $t(G)$ consists of the
root and two disjoint subtrees that are the computation trees for $t(G/e_{m})$
and $t(G-e_{m})$.
\end{proof}

\bigskip

Definition \ref{oldact} defines active edges using the structure of $G$, and
this leads to Definition \ref{act}'s description of $t(G)$ as a generating
function for maximal spanning forests. Proposition \ref{actact} shows that we
may also see activity from an algorithmic viewpoint: the external activity of
a particular $e\notin E(T)$ is revealed in the fact that one step of a
calculation of $t(G)$ involves removing $e$ using part (b) of Definition
\ref{delcont} rather than part (a). This distinction affects the result of the
computation, so it would be important even if activity could not be
conveniently described using the structure of $G$, or did not contribute to a
convenient closed form for $t(G)$.

\bigskip

Here is an analogue of Definition \ref{delcont} for the weighted interlace polynomial.

\begin{definition}
\label{alg}If $G$ is a weighted graph then $q(G)$ is determined recursively by
the following properties.

(a) If $a$ is a looped vertex then
\[
q(G)=\beta(a)q(G-a)+\alpha(a)(x-1)q(G^{a}-a).
\]

(b) If $a$ and $b$ are loopless neighbors in~$G$ then
\[
q(G)=\beta(a)q(G-a)+\alpha(a)q((G^{ab}-b)^{\prime}),
\]
where $(G^{ab}-b)^{\prime}$ is obtained from $G^{ab}-b$ by changing the
weights of $a$ to $\alpha^{\prime}(a)=\beta(b)$ and $\beta^{\prime}%
(a)=\alpha(b)(x-1)^{2}$.

(c) If $a$ is isolated and looped then $q(G)=(\alpha(a)(x-1)+\beta(a))q(G-a)$.

(d) If $a$ is isolated and unlooped then $q(G)=(\alpha(a)(y-1)+\beta
(a))q(G-a)$.

(e) The empty graph $\emptyset$ has $q(\emptyset)=1$.
\end{definition}

The preceding discussion of activities and the Tutte polynomial suggests the following.

\begin{definition}
\label{algact} A node of a computation tree representing a recursive
implementation of Definition \ref{alg} is \emph{active} if it has precisely
one child, i.e., if it represents an application of part (c) or part (d).
\end{definition}

An \textquotedblleft activities formula\textquotedblright\ for $q(G)$ arises
directly from a computation tree representing an implementation of Definition
\ref{alg}. The formula has one summand for each leaf of the computation tree
(each call to part (e) of Definition \ref{alg}), representing the product of
the coefficients contributed by the nodes in the portion of the computation
tree that gives rise to that leaf. Active and non-active nodes contribute
different coefficients.

\bigskip

If $G$ is a rooted tree then the formulas of Theorem \ref{treepoly} and its
corollaries are activities formulas.

\begin{proposition}
Let $T$ be a rooted tree with root $r$, and consider a recursive
implementation of Definition \ref{alg} structured as follows. If possible,
apply part (d) of Definition \ref{alg}; if not and there is a parent vertex
other than $r$ then apply part (b) with a leaf of the type denoted $\ell$ in
Lemma \ref{recur} as $a$; otherwise apply part (b) with the last child of $r$
as $a$. The leaves of the computation tree representing this implementation
correspond to the earlier sibling covers of $T$, with the es-cover $I$
corresponding to a given leaf constructed from the portion of the computation
that gives rise to that leaf as follows: an occurrence of part (d) of
Definition \ref{alg} contributes its $a$ to $I$, and an occurrence of the
$q((G^{ab}-b)^{\prime})$ branch of part (b) contributes its $a$ to
$I$\thinspace.
\end{proposition}

\begin{proof}
If $T$ has no vertex other than $r$ then $\{r\}$ is the only es-cover in $T$,
and the computation consists simply of a single call to part (e) of Definition
\ref{alg}. If $T$ contains no parent vertex other than $r$ and $a$ is the last
child of $r$, then the es-covers in $T$ include $V(T)-\{r\}$ and the es-covers
of $T-a$. The first step of the computation is an application of part (b) of
Definition \ref{alg} with $b=r$. The inductive hypothesis applies to $T-a$,
and the branch of the computation corresponding to $(T^{ab}-b)^{\prime
}=(T-r)^{\prime}$ consists solely of calls to part (d) because every vertex of
$T-r$ is isolated; consequently the latter part of the computation tree
contains only one leaf, corresponding to $V(T)-\{r\}$. If $T$ contains a
parent vertex other than $r$, then the first step of the computation is an
application of part (b) of Definition \ref{alg} with $a=\ell$ as in Lemma
\ref{recur}, and $b=p(\ell)$. The computation tree contains a single node
representing this first step and also two subtrees, one corresponding to
$T-a=T-\ell$ and the other corresponding to $(T^{ab}-b)^{\prime}%
=(T-p(\ell))^{\prime}$. The proposition follows inductively from Lemma
\ref{recur}.
\end{proof}

\bigskip

We do not know whether or not it is possible to reformulate Definition
\ref{algact} so that it always refers to $G$ instead of a computation tree.
Such a\ reformulation would certainly be valuable, as the resulting activities
formulas would shed light on the combinatorial significance of the interlace polynomials.

\bigskip

Definition \ref{algact} extends directly to computation trees representing
implementations of other recursions. For instance, if Definition \ref{alg} is
augmented by incorporating pendant-twin reductions then the resulting
computation trees will have active nodes representing these reductions, in
addition to active nodes representing parts (c) and (d) of Definition
\ref{alg}.

\bigskip

\textbf{Acknowledgments}

\smallskip\bigskip

We are grateful to M. Bl\"{a}ser, B. Courcelle, J. A. Ellis-Monaghan, G.
Gordon, C. Hoffmann and an anonymous referee for advice and encouragement. We
also appreciate\ the support of Lafayette College.

\end{document}